\documentclass{elsarticle}
\usepackage{amssymb,wrapfig,amsmath, amsthm, amscd,ifthen}
\usepackage[cp1251]{inputenc}
\usepackage{color} %Color definition
\usepackage{hyperref} %Internal and external links
\usepackage{graphicx,caption,subcaption,mathrsfs,appendix}
\newtheorem{theorem}{Theorem}
\newtheorem{lemma}[theorem]{Lemma}

\newcommand{\E}{\mathbb E}
\newcommand{\N}{\mathbb N}

\begin{document}
\title{$\gamma$-variable first-order logic of preferential attachment random graphs\tnoteref{t1}}
\tnotetext[t1]{Present work was funded by RFBR, project number 19-31-60021.}
\author[1]{Y.A. Malyshkin}
\ead{yury.malyshkin@mail.ru}

%\author{Yury Malyshkin, Maksim Zhukovskii}
%\address{Moscow Institute of Physics and Technology}
%\keywords{uniform attachment, random graphs, 0-1 law, first-order logic}
%\date{}

\address[1]{Moscow Institute of Physics and Technology}

\begin{abstract}
We study logical limit laws for preferential attachment random graphs. In this random graph model, vertices and edges are introduced recursively: at time $1$, we start with vertices $0,1$ and $m$ edges between them. At step $n+1$ the vertex $n+1$ is introduced together with $m$ edges joining the new vertex with $m$ vertices chosen from $1,\ldots,n$ independently with probabilities proportional to their degrees plus a positive parameter $\delta$. We prove that this random graph obeys the convergence law for first-order sentences with at most $m-2$ variables.
\end{abstract}


\maketitle

\section{Introduction}

In this paper, we study first-order (FO) logical laws for preferential attachment graphs. FO logical laws have been studied for different random graph models, such as binomial random graph $G(n,p)$ (\cite{Janson,Survey,Shelah,Spencer_Ehren}), geometrical random graphs (\cite{geo}), uniform attachment graphs (\cite{McColm,MZ21}), preferential attachment tree (\cite{MZ20}) and many others. Preferential attachment random graphs, that were introduced introduced in \cite{barabasi} by Barab{\'a}si and Albert, are widely used to model different complex networks  (see, e.g. \cite[Section 8]{Hof16}).

FO sentences about graphs are build of the following symbols: two relation  symbols, equality $=$ and adjacency $\sim$ (FO sentence does not recognize multiple edges between two vertices, just existing of an edge), variables (that take values in the vertex set), logical connectives $\wedge,\vee,\neg,\Rightarrow,\Leftrightarrow$, quantifiers $\exists,\forall$ and brackets (see the formal definition in, e.g.,~\cite{Libkin,Survey,Strange}). A series of random graphs $G_n$ on the vertex set $[n]:=\{1,...,n\}$, $n\in\N$, obeys FO zero-one law, if, for any FO sentence $\phi$, $\lim_{n\to\infty}\Pr(G_n\models\phi)\in\{0,1\}$. $G_n$ obeys FO convergence law, if, for every FO sentence $\phi$, $\Pr(G_n\models\phi)$ converges as $n\to\infty$. A primal tool for proving logical laws is the Ehrenfeucht-Fra\"{\i}ss\'{e} pebble game (see, e.g., \cite[Chapter 11.2]{Libkin}). There is a well-known relation between FO logic and pebble game.
\begin{lemma}
\label{Ehren}
Duplicator wins the $\gamma$-pebble game on $G$ and $H$ in $R$ rounds if and only if, for every FO sentence $\varphi$ with at most $\gamma$ variables and quantifier depth at most $R$, either $\varphi$ is true on both $G$ and $H$ or it is false on both graphs.
\end{lemma}
The existence of the winning strategy for Duplicator depends on a local structure of graphs. Sufficient conditions for it have been formulated in \cite{MZ21}. Let $n_0<N_0<n$ be integers, $a=3R$, $G_n$ be a graph on $[n]$. Define the properties
\begin{itemize}
\item[${\sf Q1}$] For every cycle with at most $a$ vertices one of the following two options holds:

\begin{itemize}

\item either it has all vertices inside $[N_0]$, and any path connecting $[n_0]$ with this cycle and heaving length at most $a$ has all vertices inside $[N_0]$, 
 
\item or it is at distance at least $a$ from $[n_0]$; 

\end{itemize}

every path with at most $a$ vertices joining two vertices of $[n_0]$ has all vertices inside $[N_0]$; 

and any two cycles with all vertices in $[n]\setminus[n_0]$ and having at most $a$ vertices are at distance at least $a$ from each other.

%, and every simple path of length at most $a$ joining two vertices of $[n_0]$ has all vertices inside $[N_0]$,
\item[${\sf Q2}$] For every $b\leq a$, there exist at least $m$ distinct copies of $C_b$ ($C_b$ is the cycle of length $b$) with all vertices in $[n]\setminus[N_0]$.% and, any two cycles in $[n]\setminus[N_0]$ with at most $a$ vertices are at distance at least $a$ from each other,
\item[${\sf Q3}$] Every vertex of $[N_0]$ has degree at least $N_0+m$.
\end{itemize}
\begin{lemma}(\cite[Claim 4]{MZ21})
\label{claim_m-game}
Let $H_1,H_2$ be graphs on vertex sets $[n_1]$ and $[n_2]$ respectively with minimum degrees at least $m$. Let $H_1|_{[N_0]}= H_2|_{[N_0]}$ and both $H_1$ and $H_2$ have properties ${\sf Q1}$, ${\sf Q2}$, ${\sf Q3}$. Then Duplicator wins the $(m-2)$-pebble game on $H_1$ and $H_2$ in $R$ rounds.
\end{lemma}
The implementation for this result (proven in \cite{MZ21}) is that if, for some sequence of random graphs $G_n$, for any $\epsilon>0$, there are $n_0,N_0,N$, such that conditions of Lemma~\ref{claim_m-game} hold for all $n_1,n_2>N$ with probability at least $1-\epsilon$, then $G_n$ obeys FO convergence law.
\newline

Let describe preferential attachment graph model that we consider. Let fix $m\in\N,$ $m\geq 2$. We start with complete graph $G_{1}$ that consists of vertices $0,1$ and $m$ edges between them. To build graph $G_{n+1}$ from $G_n$ we add one new vertex $n+1$ and independently draw $m$ edges from it to vertices of $G_n$, chosen from already existing vertices with probabilities proportional to their degree plus some parameter $\delta>0$, i.e., given $G_n$, probability to draw a given edge to the vertex $i$, $i=0,...,n$, equals
$$\frac{\deg_{G_n}i+\delta}{(m+\delta)n+\delta}.$$
The degree destribution of $G_n$ obeys power law with exponent $\tau=3+\frac{\delta}{m}>3$. 

Let formulate our main result.
\begin{theorem}
\label{th:main}
$G_n$ obeys  $\mathrm{FO}^{m-2}$ convergence law.
\end{theorem}

To prove this theorem we need some properties of described preferential attachment graph model.
One of the properties is that the maximal degree (we denote it as $M(n)$) grows approximayely as $n^{\frac{1}{\tau-1}}$. It could be formulated as follow
\begin{lemma}
For any $\epsilon>0$ 
$$\Pr\left(\forall n>n_0:n^{\frac{1}{\tau-1}-\epsilon}<M(n)<n^{\frac{1}{\tau-1}+\epsilon}\right)\to 1$$
as $n_0\to\infty$.
\end{lemma}
It is a standard (and well known) result so we omit its proof here (the stronger version for $m=1$ was proven in \cite{Mori05}, changing to $m>1$ does not affect the proof, which could be done as, for example, the proof of the similar result in Theorem 1, case 1 in \cite{M21}).
This lemma results in that there exists a random variable $T$, $T<\infty$ almost surely, such that no double edges would be drawn after time $T$.

The other well know property of this model (see, e.g., \cite{Mori05} for $m=1$, for $m>1$, the arguments are the same) is that, for any vertex $i$, the fraction $\frac{\deg i}{n^{\frac{1}{\tau-1}}}$ converges in distribution to a positive random variable, in particular, $\deg i$ exceeds $N_0+m$ after a random time $T_i$, $T_i<\infty$ almost surely.  

We would also need the following result about subgraphs of the preferential attachment graph. Let $H$ be a graph on $k$ vertices. Let $\pi$ be an ordering of vertices of $H$ (i.e. $\pi$ is a mapping of the vertex set of $H$ to $[k]$). Let $N_n(H,\pi)$ be the number of different subgraph, isomorphic to $H$ in $G_n$ that correspond to the ordering $\pi$ (meaning that an older vertex in $G_n$ corresponds to an older vertex in the ordering $\pi$). Define the function
$$B(H,\pi)=\max_{s=0,1,...,k}\left\{-s-\sum_{i=s+1}^k\left(\frac{1}{\tau-1}d_H^{(in)}(\pi^{-1}(i))+\left(1-\frac{1}{\tau-1}\right)d_H^{out}(\pi^{-1}(i))\right)\right\},$$
where $d_{H}^{in}(v)$ is the number of vertices $u$ of $H$, adjacent to $v$, such that $\pi(u)>\pi(v)$ and $d_{H}^{out}(v)$ is the number of vertices $u$ of $H$, adjacent to $v$, such that $\pi(u)<\pi(v)$.
Let $r(H,\pi)$ be the number of optimizers to this maximum.
Then the following result was proven in \cite{Gar19}.

\begin{theorem}(\cite[Theorem 5.1.2]{Gar19})
\label{th:subgraphs}
There exist $0<C_1(H,\pi)\leq C_2(H,\pi)<\infty$ such that
$$C_1(H,\pi)\leq\lim_{t\to\infty}\frac{\mathbb{E}(N_n(H,\pi))}{n^{k+B(H,\pi)}\ln^{r(H,\pi)-1}n}\leq C_2(H,\pi).$$
\end{theorem}
This result provides bounds on the expected number of subgraphs. If the expected number of subgraphs grows to infinity, we would need lemma, that follows from Proposition 5.1.4 in \cite{Gar19}.

\begin{lemma}
\label{lem:conc}
Consider a subgraph $H$ with an ordering $\pi$ such that $\E N_n(H,\pi)\to\infty$ as $n\to\infty$. Denote by $\mathcal{H}$ the set of all possible connected subgraphs (with ordering) composed by two distinct copies of $(H,\pi)$ (such that the ordering of corresponding subsets in new graph is the same as in the old ones) with at least one edge in common. Let $\sum_{(\hat{H},\hat{\pi})\in\hat{\mathcal{H}}}\E N_n(\hat{H},\hat{\pi})=o\left(\E(N_n(H,\pi))^2\right)$ as $n\to\infty$. Then for any constant $C>0$ we have $\Pr(N_n(H,\pi)>C)\to 1$ as $n\to\infty$.
\end{lemma}

\section{Prove of the main result}
We first use Theorem~\ref{th:subgraphs} to establish two lemmas, that we later use to prove that $G_n$ satisfies conditions of Lemma~\ref{claim_m-game} for all $n_1,n_2>N$ for some random $N<\infty$, which would result in statement of Theorem~\ref{th:main}. Recall that $\tau>3$, which we would use in the proof of the lemmas.
\begin{lemma} 
\label{lem:rare_subgraphs}
Let $H$ be a connected graph on $k$ vertices with the minimun degree at least $2$ and at least $k+1$ edges (we would call such graph rare). Then there exist $0<C_1(H,\pi)\leq C_2(H,\pi)<\infty$ such that
$$C_1(H,\pi)\leq \mathbb{E}N_n(H,\pi)\leq C_2(H,\pi).$$
\end{lemma}
\begin{proof}
Let us find $B(H,\pi)$ and $r(H,\pi)$ for such graphs. Note that, for any graph, the summation of all indegrees is always equal to the summation of all outdegrees. Moreover, the summation of outdegrees of the last $j<k$ vertices would exceed the summation of indegrees (since we consider a connected graph there would be edges drawn from them to older vertices). Also, the summation of their degrees would be at least $2j$. Since, $\tau>3$ for $0<s<k$, we get
$$-s-\sum_{i=s+1}^k\left(\frac{1}{\tau-1}d_H^{(in)}(\pi^{-1}(i))+\left(1-\frac{1}{\tau-1}\right)d_H^{out}(\pi^{-1}(i))\right)$$
$$=-s-\frac{1}{\tau-1}\sum_{i=s+1}^{k}d_H^{(in)}(\pi^{-1}(i))-\left(1-\frac{1}{\tau-1}\right)\sum_{i=s+1}^{k}d_H^{out}(\pi^{-1}(i))$$
$$<-s-\frac{1}{2}\sum_{i=s+1}^{k}d_H^{(in)}(\pi^{-1}(i))-\frac{1}{2}\sum_{i=s+1}^{k}d_H^{out}(\pi^{-1}(i))$$
$$=-s-\frac{1}{2}\sum_{i=s+1}^k\left((d_H^{(in)}(\pi^{-1}(i))+d_H^{out}(\pi^{-1}(i)))\right)$$
$$\leq -s-\frac{1}{2}2(k-s)=-k.$$
For $s=0$ (in this case the summations of outdegrees and indegrees are equal and at least $k+1$) we would get
$$-s-\sum_{i=s+1}^k\left(\frac{1}{\tau-1}d_H^{(in)}(\pi^{-1}(i))+\left(1-\frac{1}{\tau-1}\right)d_H^{out}(\pi^{-1}(i))\right)$$
$$=-\frac{1}{\tau-1}\sum_{i=1}^{k}d_H^{(in)}(\pi^{-1}(i))-\left(1-\frac{1}{\tau-1}\right)\sum_{i=1}^{k}d_H^{out}(\pi^{-1}(i))$$
$$\leq-k-1.$$
Therefore $B(H,\pi)=-k$ and the maximum is achieved on exactly one value (when $s=k$). Lemma~\ref{lem:rare_subgraphs} follows from Theorem \ref{th:subgraphs}.
\end{proof}

\begin{lemma} 
\label{lem:cycles}

Let $H$ be a cycle on $k$ vertices. Then there exist $0<C_3(H,\pi)\leq C_4(H,\pi)<\infty$ such that, for large enough $n$,
$$C_3(H,\pi)\ln n\leq \mathbb{E}N_n(H,\pi)\leq C_4(H,\pi)\ln n.$$
\end{lemma}
\begin{proof}
Once again, let find $B(H,\pi)$ and $r(H,\pi)$ for such cycles. Note that if we consider only the last $j$ (for $j<k$) vertices, then for them the summation of outdegrees exceeds the summation of indegrees and their total degree is $2j$. Since $\tau>3$ for $s<k$, similarly to Lemma~\ref{lem:rare_subgraphs}, we get
$$-s-\sum_{i=s+1}^k\left(\frac{1}{\tau-1}d_H^{(in)}(\pi^{-1}(i))+\left(1-\frac{1}{\tau-1}\right)d_H^{out}(\pi^{-1}(i))\right)$$
%$$=-s-\frac{1}{\tau-1}\sum_{i=s+1}^{k}d_H^{(in)}(\pi^{-1}(i))-\left(1-\frac{1}{\tau-1}\right)\sum_{i=s+1}^{k}d_H^{out}(\pi^{-1}(i))$$
%$$<-s-\frac{1}{2}\sum_{i=s+1}^{k}d_H^{(in)}(\pi^{-1}(i))-\frac{1}{2}\sum_{i=s+1}^{k}d_H^{out}(\pi^{-1}(i))$$
%$$=-s-\frac{1}{2}\sum_{i=s+1}^k\left((d_H^{(in)}(\pi^{-1}(i))+d_H^{out}(\pi^{-1}(i)))\right)$$
$$\leq -s-\frac{1}{2}2(k-s)=-k.$$
For $s=0$ (in this case the summations of outdegrees and indegrees are equal $k$) we would get
$$-s-\sum_{i=s+1}^k\left(\frac{1}{\tau-1}d_H^{(in)}(\pi^{-1}(i))+\left(1-\frac{1}{\tau-1}\right)d_H^{out}(\pi^{-1}(i))\right)$$
$$=-\frac{1}{\tau-1}\sum_{i=1}^{k}d_H^{(in)}(\pi^{-1}(i))-\left(1-\frac{1}{\tau-1}\right)\sum_{i=1}^{k}d_H^{out}(\pi^{-1}(i))$$
$$= -\frac{1}{\tau-1}k-\left(1-\frac{1}{\tau-1}\right)k=-k.$$
Therefore $B(H,\pi)=-k$ and the maximum is achieved on exactly two values (when $s=k$ and $s=0$). Hence $r(H,\pi)=2$. Lemma~\ref{lem:cycles} follows from Theorem \ref{th:subgraphs}.
\end{proof}

Now let check conditions of Lemma 3 from \cite{MZ21}. Fix $a\geq 3$, $a\in\mathbb{N}$.

First, we need to check that all degrees are at least $m$. Recall that there is random $T$ that no double edges are drawen after $T$. Hence for any $\epsilon>0$ there is (a not random) $T_0$, such that with a probability at least $1-\epsilon$ no double edges are drawn after $T_0$. That means that all vertices after $T_0$ have degrees at least $m$ (the degrees of vertices that appears before $T_0$ would be adressed later, as part of the third condition).

Note that the number of subgraphs isomorphic to a given graph in $G_n$ is not decreasing with $n$. Due to Lemma~\ref{lem:rare_subgraphs} the expected number of rare subgraphs in $G_n$ with at most $3a$ vertices is bounded by some constant $C_{a}$ for all $n$. By Markov's inequality, for any $\epsilon>0$ with probability at least $1-\epsilon$ the number of rare subgraphs in $G_n$ with at most $3a$ vertices is bounded from above by $C_{a}/\epsilon$ for all $n$. Since the number of subgraphs is bounded and increasing (and is natural number) there is a random moment, that is finite almost surely, such that after it the number of subgraphs stay the same. Hence, after it no rare subgraphs (with at most $3a$ vertices) would be formed. Therefore, for any $\epsilon>0$, there is $n_0\in\N$, $n_0>T_0$ such that with probability at least $1-\epsilon$ there are no rare subgraphs with at most $3a$ vertices outside of $[n_0]$. Since two cycles of length at most $a$ at distance at most $a$ from each other form (together with a shortest path between them) a rare graph on at most $3a$ vertices, all cycles with all vertices in $[n]\setminus[n_0]$ and having at most $a$ vertices are at distance at least $a$ from each other. That provides us with the third part of condition Q1.

To get the first and the second parts, consider subgraphs that contain two cycles of length at most $a$ that are connected by a path of length at most $n_0+a$. These graphs satisfy conditions of Lemma~\ref{lem:rare_subgraphs} and hence the expected number of such graphs is bounded from above by a constant. Therefore (as in previous argument) there is a random moment, after which no such subgraphs would be formed. As result, for any $\epsilon>0$, there is $N_0>n_0$, such that with probability $1-\epsilon$ there are no vertices outside of $[N_0]$ that belong to such subgraphs. Since $[n_0]$ induces a connected subgraph and contains a cycle of length $3$, that means that any cycle of length at most $a$ is either at distance at least $a$ from $[n_0]$ or completly inside $[N_0]$. Also any path connecting cycle in $[N_0]$ with $[n_0]$ that have length at most $a$ has all vertices in $[N_0]$. Moreover, any two cycles of length at most $a$ that are outside of $[N_0]$ have distance at least $n_0+a+1>a$ between them. Hence we check the condition Q1 of Lemma~\ref{claim_m-game}.

Let us check the second condition. Due to Lemma~\ref{lem:cycles} for any $b\leq a$ we get that $\E N_n(C_b,\pi)\to \infty$. Also, any connected graph $\hat{H}$ composed of two distinct cycles is a rare graph and hence, as was mentioned above, for any $\epsilon>0$ there is a constant $C_{\epsilon}(\hat{H})$, such that $N_n(H)<C$ with probability at least $1-\epsilon$ for all $n$. Hence, conditions of Lemma~\ref{lem:conc} hold and $\Pr(N_n(C_b,\pi)>m+N_0^b+N_0)\to \infty$ as $n\to\infty$. There are at most $N_0^b$ cycles on $b$ vertices that are completly inside $[N_0]$. Hence, at least $N_0+m$ of cycles of length $b$ have at least $1$ vertex outside of $[N_0]$ and, by the definition of $N_0$, they are at a distance at least $a$ from each other. Note that there could be at most $N_0$ cycles that do not share a vertex with each other but share a vertex with $[N_0]$.  Therefore at least $m$ of cycle of length $b$ have no vertices in $[N_0]$, which results in Q2.

The last condition holds due to the fact that the degree of any given vertex $i$ exceeds $N_0+m$ after some random moment $T_i$ and hence for any $\epsilon>0$ there is (a not random) $N>N_0$, such that with probability at least $1-\epsilon$ the degrees of the first $N_0$ (in particular, the first $T_0$ since $N_0>n_0>T_0$) vertices would exceed $N_0+m$ in $G_n$.

Hence for any $\epsilon>0$ there exist $N_0,N,n_0$ such that the conditions of  Lemma~\ref{claim_m-game} hold for any $n_1,n_2>N$ with probability at least $1-\epsilon$, which proves our theorem.

\section*{Acknowledgements.}
The study was funded by RFBR, project number 19-31-60021. The author is grateful to Maksim Zhukovskii for helpful discussions.

\end{document}